\documentclass[11pt]{amsart}

\usepackage{latexsym}
\usepackage{ulem}
\usepackage{amsthm}
\usepackage{amsmath}
\usepackage{amssymb}
\usepackage{color}
\usepackage{url}
\usepackage[shortlabels]{enumitem}
\newtheorem{theorem}{Theorem}

\newtheorem{lem}[theorem]{Lemma}

\newtheorem{rmk}[theorem]{Remark}

\newtheorem*{theorem*}{Theorem}

\def \n{\noindent }
\def \bs{\bigskip}
\def \P{{\mathcal P}}
\def \ra{\right\rangle}
\def \la{\left\langle}

\def \ga{\gamma}
\def \al{\alpha}
\def \be{\beta}
\def \de{\delta}

\def \bea{\begin{eqnarray*}}
\def \eea{\end{eqnarray*}}

\def \F{{\mathbb F}}
\def \dim{\mbox{dim}}
\def \avsp{\mbox{\bf avsp}}
\def \avsps{\mbox{\bf avsps}}
\def \o{\mbox{\bf 0} }
\begin{document} 

 
\title{A note on Tight Irreducible Affine Spreads}
\author{Fusun Akman}
\author{Papa Sissokho}
\address[Fusun Akman and Papa Sissokho]{Mathematics Department, Illinois State University, Normal, Illinois 61790, USA}
\email{akmanf@ilstu.edu}
\email{psissok@ilstu.edu}

\begin{abstract}
 Let $\F_q^n$ denote the vector space of  dimension $n$ over $\F_q$ and AG$(n,q)$ denote the corresponding affine space. An {\it affine vector space partition} of AG$(n,q)$ is a collection $\P$ of affine subspaces that partition the points of AG$(n,q)$. If all subspaces in $\P$ have the same dimension $d$, then $\P$ is called an {\it affine $d$-spread}. We say that an affine partition $\P$ is {\it completely tight} if for any pair $C,C'\in\P$ with $C=v+S$, $C'=v'+S'$, where $v,v'\in{\rm AG}(n,q)$ and  $S\neq S'$ are linear subspaces of $\F_q^n$, we have $ S\cap S'=\{\o\}$. 

 An affine partition $\P$ is said to be {\it irreducible}  if there is no subset $\P'\subset \P$ such that $1<|\P'|<|\P|$ and the union of all  subspaces in $\P'$ is a  subspace of AG$(n,q)$. For  all $d \geq 1$ and $n>2d$, we construct a completely tight irreducible affine $d$-spread of AG$(n,q)$. This also settles a recent conjecture of Bamberg et al.~\cite{BFIK} on the existence of tight irreducible affine $d$-spreads.

\bs\n \textbf{Keywords.} affine vector space partition, affine spread, subcube partition

\bs\n \textbf{AMS Mathematics Subject Classification.} 51E23, 05B40 
\end{abstract}
\ 
\maketitle

\section{Introduction}
 
Let $\F_q^n$ denote the vector space of  dimension $n$ over  the finite field $\F_q$ and AG$(n,q)$ denote the corresponding affine space whose subspaces are the linear subspaces of $\F_q^n$ and their cosets. An {\it affine vector space partition} (\avsp) of AG$(n,q)$ is a collection $\P$ of subspaces of AG$(n,q)$ such that each point of AG$(n,q)$ (or vector of  $\F_q^n$) is in exactly one subspace of $\P$.  The dimension of $C\in\P$ is the algebraic dimension of the linear subspace $S$ such that $C=a+S$ for some $a\in\F_q^n$. In this case, we say  $S$ is the {\it supporting linear subspace} of $C$. In particular, if $\dim\, C=d$, then $|C|=q^d$. 
If all subspaces in a partition $\P$ have equal dimension $d$, then $\P$ is called an {\it affine $d$-spread}. 
 The definition of an affine partition does not compel the supporting subspaces to be distinct, but that will be the case in the main construction of this paper.  Affine vector space partitions (\avsps) are closely related to the well-studied {\it vector space partitions}, which are collections of (projective) subspaces that partition the points of GF$(n,q)$ (see Heden~\cite{He} for a survey).  In this paper, we are interested in \avsps~$\P=\{C_1,\ldots,C_r\}$ that satisfy a subset of the following properties:
\begin{enumerate}
\item[(i)] {\bf Tightness:}
$\P$ is {\it tight} if $\bigcap_{i=1}^r S_i=\{\o\}$, where $\o$ denotes the the zero vector of $\F_q^n$ and $S_i$'s are the supporting subspaces of the $C_i$'s. The notion of tightness (see Bamberg et al.~\cite{BFIK}) was introduced under the name {\it primitivity} by Agievich~\cite{Ag} and later called {\it A-primitive} by Tarannikov~\cite{Ta}.
\item[(ii)] {\bf Complete Tightness:}
$\P$ is {\it completely tight} if $S_i\cap S_j=\{\o\}$ for any $C_i,C_j\in\P$ with $S_i\not=S_j$.  Our notion of complete tightness is motivated by the study of {\it Steiner coset partitions~\cite{AS3}, where the supporting subgroups of the cosets are required to be distinct}.
\item[(iii)] {\bf Irreducibility}:  $\P$ is {\it irreducible}  if there is no subset $\P'\subset \P$ with $1<|\P'|<|\P|$ such that the union of  all subspaces in $\P'$ is a  subspace of AG$(n,q)$. 
\end{enumerate}
Agievich~\cite{Ag} used tight \avsps~to construct structures called {\it bent rectangles}, which are used to build and classify standard {\it bent functions}. Tarannikov~\cite{Ta} derived necessary conditions for the existence of tight  affine $d$-spreads. More recently,  Bamberg, Filmus, Ihringer, and Kurz~\cite{BFIK} started a systematic study of tight and irreducible \avsps~and pioneered several directions of research on these objects. For $q=2$, \avsps~are related to {\it subcube partitions}, which are considered by Filmus et al.~\cite{FHKIRSV}. For recent results on affine spreads of quadrics, see Gupta and Pavese~\cite{GP}.

In this paper, we prove the following theorem (Theorem~\ref{thm:1}), which also verifies a conjecture~\cite[Conjecture~1]{BFIK} of Bamberg et al.
\begin{theorem}\label{thm:1}
If $d \geq 1$ and $n>2d$, then there exists a completely tight irreducible affine $d$-spread of AG$(n,q)$.  
\end{theorem}
A construction of such \avsps~is given in Section~\ref{sec:prf-main}.
\begin{rmk}

\n $(1)$  Bamberg et al.~\cite{BFIK} have shown the existence of tight irreducible affine $d$-spreads of AG$(n,q)$ for $n\in \{4,5,6\}$ and $q=2$ as well as for $d=\lfloor n/2 \rfloor$ and any prime power $q$. 

\n $(2)$  Theorem~\ref{thm:1} is stronger than the conjecture~\cite[Conjecture~1]{BFIK} of Bamberg et al.,  in that it does not require ``complete tightness'' or that ``$d$ divides $n$''.

\n $(3)$ The conjecture \cite[Conjecture~1]{BFIK} of Bamberg et al. uses an equivalent projective formulation. Thus, the dimension $k$ there  corresponds to $d+1$ here. 
\end{rmk}
%
\section{Proof of Theorem~\ref{thm:1}}\label{sec:prf-main}
In this section, we provide a constructive proof of the main theorem (Theorem~\ref{thm:1}), i.e., a construction of completely tight \avsps~of  AG$(n,q)$ that are affine $d$-spreads.  Let $\la X\ra$ denote the linear span of  any $X\subseteq \F_q^n$, $d$ and $n$ be positive integers such that $n>2d$, and $q$ be a prime power.

\bs\n {\bf Construction~1:} 
\begin{enumerate}
\item The finite field $H=\F_{q^{n-d}}$ is an $(n-d)$-dimensional vector space over $\F_q$, and $W$ is a $d$-dimensional subspace of $H$ with basis $\{w_1,\ldots,w_d\}$. 
\item $U$ is a $d$-dimensional vector space over $\F_q$ with basis $\{u_1,\ldots,u_d\}$.
\item Let $V=H\oplus U$, an $n$-dimensional vector space over $\F_q$. For each $\ga\in H$, define 
\[S_\ga:=\la \ga w_1+u_1,\ldots,\ga w_d+u_d\ra,\]
where $\ga w_i$ denotes field multiplication in $H$.
\item Pick some $\al\in (H\setminus W)$, and define 
\[C_{\al,\ga}:=\al\ga+S_\ga=\{\al\ga+x:\; x\in S_\ga\}\]
and
\[\P_{\al}:=\{C_{\al,\ga}:\; \ga\in H\}.\]
\end{enumerate}
\begin{lem}\label{lem:1}
The set $\P_\al$  given in Construction~1 is  a completely tight affine $d$-spread of  AG$(n,q)$.
\end{lem}
\begin{proof}
It follows from Construction~1 that $S_\ga$ is a linear subspace of $V=H\oplus U\cong \F_q^{n}$ for any $\ga\in H$, and $C_{\al,\ga}:=\al\ga+S_\ga$ is an affine subspace of AG$(n,q)$. It is known from~\cite[Theorem~3.1]{ESSSV} that $S_\ga\cap S_{\ga'}=\{\o\}$  for $\ga,\ga'\in H$ with $\ga\not=\ga'$, but for completeness, we provide the short proof here.  Let
\bea
x\in S_\ga\cap S_{\ga'}
&\Longleftrightarrow& \sum_{i=1}^d a_i(\ga w_i+u_i)=x=\sum_{i=1}^d b_i (\ga' w_i+u_i) \quad \mbox{(for some scalars $a_i,b_i\in \F_q$)}\cr
&\Longleftrightarrow& \sum_{i=1}^d (a_i\ga-b_i\ga')w_i=\sum_{i=1}^d (b_i-a_i)u_i\cr
&\Longrightarrow& a_i-b_i=0 \mbox{ for $1\leq i\leq d$ } 
\quad  (H \cap U=\{\o\})\cr
&\Longrightarrow& (\ga-\ga')\sum_{i=1}^d b_i w_i=\o\cr
&\Longrightarrow& \sum_{i=1}^d b_i w_i=\o \quad (\ga-\ga'\not=\o)\cr
&\Longrightarrow&  a_i=b_i=0 \mbox{ for $1\leq i\leq d$ } \quad \mbox{($\{w_1,\ldots ,w_d\}$ is a basis of  $W$)}.\cr
\eea
Then  $x=\sum_{i=1}^d a_i(\ga w_i+u_i)=\o$, i.e., $S_\ga\cap S_{\ga'}=\{\o\}$. Thus, to prove the lemma, it remains to show that 
$C_{\al,\ga}\cap C_{\al,\ga'}=\emptyset$, because then $\P$ covers all $q^d\cdot q^{n-d}=q^n$ points of AG$(n,q)$. For a proof by contradiction, suppose $x\in C_{\al,\ga}\cap C_{\al,\ga'}$. Therefore, we  have
\begin{align}
&\al\ga+\sum_{i=1}^d a_i(\ga w_i+u_i)=
x=\al\ga'+\sum_{i=1}^d b_i (\ga' w_i+u_i) \quad \mbox{(for some scalars $a_i,b_i\in \F_q$)}\cr
&\Longleftrightarrow \al(\ga-\ga')+ \sum_{i=1}^d (a_i\ga-b_i\ga')w_i=\sum_{i=1}^d (b_i-a_i)u_i\label{eq:a1}\\
&\Longrightarrow a_i-b_i=0 \mbox{ for $1\leq i\leq d$ } \quad (H \cap U=\{\o\})\cr 
&\Longrightarrow (\ga-\ga')\left(\al+ \sum_{i=1}^d b_iw_i\right)=\o \cr
&\Longrightarrow \al+ \sum_{i=1}^d b_iw_i=\o \quad (\ga-\ga'\not=\o)\cr
&\Longrightarrow  b_i=a_i=0 \mbox{ for $1\leq i\leq d$ } \quad \mbox{(since $\al\not\in W$ and $\{w_1,\ldots ,w_d\}$ is a basis of  $W$)}\label{eq:a2}\\
&\Longrightarrow  \al\ga=\al\ga' \cr
&\Longrightarrow  \ga=\ga' \quad \mbox{(by~\eqref{eq:a1},~\eqref{eq:a2},  and since $\al\not=\o$)},\notag
\end{align}
which is a contradiction as $\ga\not=\ga'$. 
\end{proof}

\bs Next, we adapt Construction~1 to derive a construction for a completely tight irreducible \avsp.

\bs\n {\bf Construction~2:} 
\begin{enumerate}
\item For $H,U$ as in Construction~1, we further let $\be$ be a primitive element for $H$ over $\F_q$ and choose $W=\la \be^i:\; 1\leq i\leq d \ra$. Since $n-d>d$, we have $\be^0=1_H\not\in W$. 
\item Again let $V=H\oplus U\cong \F_q^{n}$ and, for each $\ga\in H$, define 
\[S_\ga:=\la \ga \be+u_1,\ldots,\ga \be^{d}+u_d\ra,\]
where $\ga \be^i$ denotes field multiplication in $\F_{q^{n-d}}$.
\item Let $\al=1=1_H\in H\setminus W$ and  denote  $C_{\al,\ga}$ by $C_{\ga}=1\cdot \ga+S_\ga=\ga+S_\ga$. Thus, our partition is
\[\P=\{C_{\ga}:\; \ga\in H\}.\]
\end{enumerate}
\begin{theorem}\label{thm:2}
Let $\P$ be the affine partition given in Construction~2. Then $\P$ is a completely tight irreducible affine $d$-spread of AG$(n,q)$.
\end{theorem}
\begin{proof}{Proof of Theorem \ref{thm:1}}
Since Construction~2 retains all the relevant properties of Construction~1, the fact that $\P=\{C_{\ga}:\; \ga\in H\}$ is a completely tight \avsp~follows from Lemma~\ref{lem:1}.

Next, we prove the irreducibility of $\P$. Assume that there exists some sub-partition $\P'\subset \P$ such that $1<|\P'|<|\P|$ and the union of all subspaces in $\P'$ is a subspace $V'$ of AG$(n,q)$. 
By translating\footnote{(Complete) tightness and irreducibility of $\P_\al$'s are preserved by vector translations.} $\P$ if necessary, we can assume that $V'$ is linear. Thus, if $\P$ is reducible, then we have a proper subspace $V'$ of $V$ such that
\begin{equation}\label{eq:span}
V'=\la \bigcup_{C\in \P'} C\ra=\bigcup_{C\in \P'} C.
\end{equation}
Since $\P'$ contains at least two cosets and $V'$ is a linear subspace,  one of the minimum of two cosets must be $C_{0}=S_0=U$ and the other one $C_{\ga}$, for some $\ga\not=0$.
Then it follows from~\eqref{eq:span} that
\[U=\la u_1, \ldots, u_d\ra \subseteq V'\mbox{ and } \la C_{\ga}\ra = \la \ga, \ga\be+u_1,\ldots,\ga\be^d+u_d\ra\subseteq V',\]
and hence,
 \[\la \ga,\ga\be,\ldots,\ga\be^d, u_1, \ldots, u_d\ra \subseteq V'.\]
 We now show by induction that $\ga \be^k\in V'$ for all integers $k\geq1$.
Given that $\ga \be\in V'$, the base case holds. Next, assume that $\ga \be^k\in V'$ for any $k\geq1$. Since $\ga\in V'$ and $V'$ is a linear subspace, we must have $\ga+\ga \be^k\in V'$. Thus, 
it follows from~\eqref{eq:span} that there exists $\de\in H$ such that  $C_\de\in\P'$ and $\ga+\ga \be^k \in C_\de$. Consequently, we have
\begin{align*}\label{eq:cond1}
\ga+\ga \be^k \in C_\de 
& \Longleftrightarrow 
\ga+\ga \be^k = \de+\sum_{i=1}^d a_i (\de\be^i+u_i) \quad \mbox{(for some scalars $a_i\in \F_q$)}\cr 
& \Longleftrightarrow 
-\de+\ga+\ga \be^k -\sum_{i=1}^d a_i \de\be^i=\sum_{i=1}^d a_i u_i  \cr 
& \Longrightarrow 
 \mbox{ $a_i=0$ for $1\leq i\leq d$ } \quad (H\cap U=\{ \o\}) 
 \cr
 & \Longrightarrow 
\de=\ga+\ga \be^k.
\end{align*}

 \[C_\de= C_{\ga+\ga \be^k}=\ga+\ga \be^k+ \la (\ga+\ga \be^k)\be+u_1,\ldots,(\ga+\ga \be^k)\be^d+u_d\ra \subseteq V'.\]
In particular,  it must be true that
\begin{align}\label{eq:cond2}
\ga+\ga \be^k+\left((\ga+\ga \be^k)\be+u_1\right) =
\ga+\ga\be+\ga \be^k+\ga \be^{k+1}+u_1 \mbox{ is in } V'\Longrightarrow \ga \be^{k+1}\in V' .
\end{align}
This shows the inductive step and proves that $\ga \be^k\in V' $ for any integer $k\geq 1$.
Since $\ga\not=0$ and $H=\F_q(\beta) \cong \F_{q^{n-d}}$, it follows that 
\[\{\ga \be^k:\; k\geq 1\}=\ga\F_{q^{n-d}}^*=\F_{q^{n-d}}^*\subseteq V',\]
where $\F_{q^{n-d}}^*=\F_{q^{n-d}}\setminus\{\o\}$. 
Moreover $U\subseteq V'$, and thus, $V=H\oplus U\subseteq V'$, which implies that $V'=V$. We have proven the irreducibility of $\P=\{C_\ga:\; \ga\in H\}$ and the proof is complete.
\end{proof} 

 
\end{document}